\documentclass[11pt, reqno]{amsart}

\usepackage{amsmath,amssymb,amsfonts,amsthm}
\usepackage{mathrsfs}
\usepackage{enumitem}
\usepackage[usenames]{color}
\usepackage{hyperref}

\allowdisplaybreaks
\numberwithin{equation}{section}

\theoremstyle{definition}
\newtheorem{definition}{Definition}[section]

\theoremstyle{remark}
\newtheorem{remark}[definition]{Remark}

\theoremstyle{plain}
\newtheorem{proposition}[definition]{Proposition}
\newtheorem{theorem}[definition]{Theorem}
\newtheorem{lemma}[definition]{Lemma}
\newtheorem{result}[definition]{Result}
\newtheorem{corollary}[definition]{Corollary}
\newtheorem{conjecture}[definition]{Conjecture}

\newtheorem{innercustomthm}{Result}
\newenvironment{customthm}[1]{\renewcommand\theinnercustomthm{#1}\innercustomthm}{\endinnercustomthm}

\definecolor{Wine}{rgb}{0.5,0,0.05}
\definecolor{DPurple}{rgb}{0.46,0.2,0.69}
\definecolor{Red}{rgb}{1,0,0}

\newcommand{\koba}{\mathsf{k}}
\newcommand{\dkoba}{\kappa}

\newcommand{\rprt}{\mathsf{Re}}
\newcommand{\iprt}{\mathsf{Im}}
\newcommand{\distance}{\mathrm{dist}}
\newcommand{\dtb}[1]{\delta_{#1}}
\newcommand{\bcdot}{\boldsymbol{\cdot}}
\newcommand{\bdy}{\partial}

\newcommand{\wt}{\widetilde}

\newcommand*{\defeq}{\mathrel{\vcenter{\baselineskip0.5ex \lineskiplimit0pt \hbox{\scriptsize.}\hbox{\scriptsize.}}}=}
\newcommand*{\defines}{=\mathrel{\vcenter{\baselineskip0.5ex \lineskiplimit0pt \hbox{\scriptsize.}\hbox{\scriptsize.}}}}

\newcommand{\cmpl}{\mathsf{c}}

\newcommand{\multByi}{\mathbb{J}}

\newcommand{\OM}{\Omega}
\newcommand{\unitdisk}{\mathbb{D}}

\newcommand{\geodModDomB}{\mathcal{D}(\alpha,\tau)}

\newcommand{\smoo}{\mathcal{C}}

\newcommand{\xiz}{{}^{\raisebox{-1pt}{$\scriptstyle{\xi\!}$}}z}

\newcommand{\C}{\mathbb{C}} 
\newcommand{\R}{\mathbb{R}}

\begin{document}
\title{On the continuous extension of Kobayashi isometries}
\author{Anwoy Maitra}
\address{Department of Mathematics, Indian Institute of Science, Bangalore 560012, India}
\email{anwoymaitra@iisc.ac.in}

\begin{abstract}
	We provide a sufficient condition for the continuous extension of isometries for the Kobayashi distance
	between bounded convex domains in complex Euclidean spaces having boundaries that are only
	slightly more regular than $\smoo^1$. This is a generalization of a recent result by A.~Zimmer.
\end{abstract}


\maketitle

\vspace{-8mm}
\section{Introduction}

In this paper, we provide a sufficient condition for the continuous extension to $\overline{\OM}_1$ of isometries, with respect to the
Kobayashi distances on $\OM_1$ and $\OM_2$, between a pair of bounded convex domains $\OM_1$ and $\OM_2$ in complex
Euclidean spaces (of not necessarily the same dimension). In this setting it is well
known that such isometries \emph{do} exist. A consequence of fundamental work by Lempert \cite{Lemp_Kob_Met,
Lemp_Holo_Ret_Int_Met} is that if $\OM \subset \C^n$ is a bounded convex domain, then given a pair
of distinct points $z_1,z_2 \in \OM$, there exists a \emph{holomorphic} map $F: \unitdisk \to \OM$ that is an isometry
with respect to the Kobayashi distances on $\unitdisk$ and $\OM$ and such that $z_1,z_2 \in F(\unitdisk)$. We
call such a map a \emph{complex geodesic} of $\OM$ through $z_1$ and $z_2$.

\smallskip 

The question of whether a complex
geodesic extends continuously to $\overline{\unitdisk}$ is not an easy one. The earliest result
in this direction was given by Lempert \cite{Lemp_Kob_Met}, which states that if
$\OM \subset \C^n$ is strongly convex with $\smoo^k$-smooth boundary, $k \geqslant 2$, then every
complex geodesic $F: \unitdisk \to \OM$ extends to a $\smoo^{k-2}$-smooth mapping on $\overline{\unitdisk}$ (by a
$\smoo^0$-smooth mapping we mean a continuous one).
Since then, there has been a number of works dealing with the continuous (or smooth) extension of complex geodesics;
see \cite{Aba_Bdy_Bhvr_Inv_Dist,Mer_Comp_Geo_It,Bha_Comp_Geo,Zim_CharDomLimAut}.

\smallskip

While Lempert's result might suggest that the boundary regularity of the target convex domain $\OM$ controls the boundary behaviour of a
complex geodesic of $\OM$, that is not the case\,---\,see \cite[Remark~1.8]{Huang_Bdy_Rig_Hol} and \cite[Example~1.2]{Bha_Comp_Geo}. The
latter example shows that there exist $\smoo^{\infty}$-smoothly bounded convex domains having complex geodesics that do not extend
continuously to $\overline{\unitdisk}$. In view of this, the question of $\smoo^0$-extension of Kobayashi isometries \emph{in general} is
certainly a challenging one.

\smallskip

Before we state the main result of this paper, let us look at the motivations behind it. Our chief motivation is the following recent
result by Zimmer:

\begin{result}[Zimmer {\cite[Theorem~2.18]{Zim_CharDomLimAut}}] \label{R:Zim_Cont_Ext_Isom}
	Let $\OM_j \subset \C^{n_j}$, $j=1,2$, be bounded convex domains with $\smoo^{1,\alpha}$-smooth boundaries,
	where $\alpha \in (0,1)$.
	Suppose that $\OM_2$ is $\C$-strictly convex. Let $F: \OM_1 \to \OM_2$ be an isometric embedding with respect to the Kobayashi
	distances. Then $F$ extends to a continuous map $\wt{F}: \overline{\OM}_1 \to \overline{\OM}_2$.	
\end{result}

\noindent{Recall that for a convex, $\smoo^1$-smoothly bounded 
domain $\OM \subset \C^n$, to be \emph{$\C$-strictly convex} means that for every $\xi \in \bdy \OM$, 
\[ 
  \big( \xi + T^{\C}_{\xi}(\bdy \OM) \big) \cap \overline{\OM} = \{\xi\},
\]
where $T^{\C}_{\xi}(\bdy \OM)$ denotes the complex tangent space to $\bdy \OM$ at $\xi$, given by $T_{\xi}(\bdy \OM) \cap i T_{\xi}(\bdy
\OM)$, and where we view $T_{\xi}(\bdy \OM)$ \emph{extrinsically} as a real hyperplane in $\R^{2n} \equiv \C^n$ (also see 
Section~\ref{S:Some_facts_convex_domains}).}

\smallskip

A close reading of the proof of the above result reveals that it actually establishes a stronger result. Before we can state this result,
we need to fix some pieces of notation. The first set of notations pertains to the real category. For $U \subset \R^d$
an open set and $f:U \to \R$ a $\smoo^1$-smooth function, $Df$ will denote the total derivative of $f$; it is a continuous mapping from
$U$ into $\mathscr{L}(\R^d,\R)$. For a vector $v \in \R^d$ with $\|v\|=1$, $D_v$ will denote the directional derivative in the direction
of $v$.

\smallskip

In what follows, we shall identify $\C^n$ with $\R^{2n}$ in the following manner:
\[
  \C^n \ni z = (z_1,\ldots,z_n) \longleftrightarrow (\rprt(z_1),\iprt(z_1),\rprt(z_2),\iprt(z_2),\ldots,\rprt(z_n),\iprt(z_n)) \in
  \R^{2n}.
\]
We let $\multByi$ denote multiplication by $i$ in $\C^n$ regarded as an $\R$-linear map from $\C^n$ to itself. In terms of the above
identification,
\[
  \multByi(x_1,\ldots, x_{2n}) = (-x_2,x_1,\ldots,-x_{2n},x_{2n-1}) \quad \forall \, (x_1,\ldots, x_{2n}) \in
  \R^{2n}.
\]
Given $a \in \C^n$ and $r > 0$, $B^{(n)}(a,r)$ will denote the open Euclidean ball in $\C^n$ with centre $a$ and
radius $r$.

We are now in a position to state the above-mentioned result. In this result, for any $\xi \in \bdy \OM$, $\eta^{\xi}$ will denote the
unit inward-pointing normal to $\bdy \OM$ at $\xi$. 

\begin{customthm}
{{\ref{R:Zim_Cont_Ext_Isom}}$\boldsymbol{{}^{\!\prime}}$}
[follows from the proof of {\cite[Theorem~2.18]{Zim_CharDomLimAut}}]
\label{R:Zim_Cont_Ext_Isom_Inf}
	Let $\Omega_j \subset \C^{n_j}$, $j=1,2$, be bounded convex domains with $\smoo^1$-smooth boundaries. 
	Suppose that there exist a constant $r > 0$, an $\alpha \in (0,1)$, and, for each $j=1,2$, a defining function $\rho_j$ for
	$\Omega_j$ such that for each $\xi \in \bdy \Omega_j$, the directional derivative $D_{\multByi\left( \eta^{\xi} \right)}\rho_j$ is 
	$\alpha$-H{\"o}lder-continuous on the ball $B^{(n_j)}(\xi,r)$. If $\Omega_2$ is $\C$-strictly convex, then every isometric embedding
	$F: \Omega_1 \to \Omega_2$ with respect to the Kobayashi distances extends to a continuous map $\widetilde{F} : \overline{\Omega}_1
	\to \overline{\Omega}_2$.
\end{customthm} 

If $\OM$ is $\smoo^1$-smoothly bounded, $\bdy\OM\ni \xi \mapsto \multByi(\eta^{\xi})$ is what is sometimes called the
\emph{complex-normal vector field} on $\bdy \OM$. The geometrical significance of the hypothesis in the above result
is as follows: one does not require $\bdy \OM_j$ to be a $\smoo^{1,\alpha}$-smooth manifold, $j=1,2$, for the
conclusion of Result~\ref{R:Zim_Cont_Ext_Isom} to hold true; it suffices to control the behaviour of $\bdy \OM_1$
and $\bdy \OM_2$ in the complex-normal directions. As stated earlier, the proof of Result~${\ref{R:Zim_Cont_Ext_Isom}}^{\!\prime}$ 
follows from a careful reading of the proof of \cite[Theorem~2.18]{Zim_CharDomLimAut} (and we shall see the required ingredients in the
proof of our main theorem). 

\smallskip

All of this raises the question whether the conclusion of the above results holds true under even lower regularity of $\bdy \OM_j$,
$j=1,2$. This question is also suggested by a related result in \cite{Bha_Comp_Geo} in which certain convex domains with
just $\smoo^1$-smooth boundaries are considered (which we shall see below). This is the second motivation for our result. But first, we
need a definition. 

\begin{definition}
	We say that a Lebesgue-measurable function $g: [0,\epsilon_0) \to [0,\infty)$, where $\epsilon_0>0$, \emph{satisfies a Dini condition}
	if 
	\[
	   \int_0^{\epsilon_0} \frac{g(t)}{t} dt < \infty.
	\]	
\end{definition}

Our main theorem (whose relation to Result~\ref{R:Zim_Cont_Ext_Isom}\,---\,via
Result~$\ref{R:Zim_Cont_Ext_Isom}^{{\!\prime}}$\,---\,is clear) is:

\begin{theorem} \label{T:contExtIsom}
	Let $\Omega_j \subset \C^{n_j}$, $j=1,2$, be bounded convex domains with $\smoo^1$-smooth boundaries.
	Suppose that there exist a constant $r > 0$ and, for each $j=1,2$, a defining function $\rho_j$ for $\Omega_j$ such that
	for each $\xi \in \bdy \Omega_j$, the directional derivative $D_{\multByi\left( \eta^{\xi} \right)}\rho_j$ has modulus of
	continuity $\omega$ on the ball $B^{(n_j)}(\xi,r)$. Assume that $\omega$ satisfies a Dini condition. If $\Omega_2$ is $\C$-strictly
	convex, then every isometric embedding $F: \Omega_1 \to \Omega_2$ with respect to the Kobayashi distances extends to a continuous map
	$\widetilde{F} : \overline{\Omega}_1 \to \overline{\Omega}_2$. 
\end{theorem}

In view of our discussion on complex geodesics above, we have the following immediate corollary to Theorem~\ref{T:contExtIsom}:

\begin{corollary}
	Let $\OM \subset \C^n$ satisfy the conditions on $\OM_2$ of Theorem~\ref{T:contExtIsom}.
	Then every complex geodesic of $\OM$ extends continuously to $\overline{\unitdisk}$.
\end{corollary}

We note that there exist plenty of functions on intervals of the form $[0,\epsilon_0)$ that 
satisfy a Dini condition but which are not $\alpha$-H{\"o}lder-continuous for \emph{any} $\alpha \in (0,1)$; examples are
the functions
\begin{equation}\label{E:Dini_example}
f_{\epsilon}(x) \defeq \begin{cases}
\dfrac{1}{|\log x|^{1+\epsilon}}, &\text{if } x \in (0,1), \\
{} & {} \\
0,                                &\text{if } x=0,
\end{cases} 
\end{equation}
for arbitrary $\epsilon>0$. While Theorem~\ref{T:contExtIsom} generalizes Result~\ref{R:Zim_Cont_Ext_Isom}, what is perhaps
more suggestive are the geometric insights that its proof reveals. Firstly, given bounded convex domains $\OM_1$ and
$\OM_2$ with $\smoo^1$-smooth boundaries, given an isometry $F: \OM_1 \to \OM_2$ with respect to the Kobayashi distances,
and given any point $\xi \in \bdy \OM_2$, how $\bdy \OM_2$ behaves in the \emph{complex-tangential} directions is largely
immaterial to the existence of a continuous extension of $F$ to $\overline{\OM}_1$, owing to adequate control on the local
geometry of $\bdy \OM_2$ at $\xi$ conferred by $\C$-strict convexity. Secondly, some elements of our proof
reveal a certain bound for the Kobayashi distance that might be of independent interest. For greater clarity,
Proposition~\ref{P:Kob_dist_ub} will present the above-mentioned bound for a special case (see
Proposition~\ref{P:Kob_dist_gen_ub} later for the more general result). We need a
definition: we say that a domain $\OM\varsubsetneq \C^n$ has $\smoo^{1,\,{\rm Dini}}$ boundary if $\bdy \OM$ is,
near each $\xi \in \bdy \OM$, the graph (relative to a coordinate chart around $\xi$)
of a $\smoo^1$ function  whose partial derivatives are
Dini-continuous (i.e., have moduli of continuity that satisfy a Dini condition). With this definition, we have:

\begin{proposition}\label{P:Kob_dist_ub}
	Let $\OM\subset \C^n$ be a bounded convex domain with $\smoo^{1,\,{\rm Dini}}$ boundary.
	Let $z_0\in \OM$. Then, there exists a constant $C>0$ such that
	\[
	  \koba_{\OM}(z_0, z) \leqslant C + \frac{1}{2}\log\left( \frac{1}{\distance(z,\OM^{\cmpl})} \right)
	  \quad \forall z\in \OM.
	\]
\end{proposition}

\noindent{The above estimate is easy to deduce for domains with $\smoo^2$-smooth boundaries. 
For domains with $\smoo^{1,\alpha}$-smooth boundaries, it was established by 
Forstneric--Rosay \cite{Forst_Ros_LocKobMet}. In view of \eqref{E:Dini_example},
Proposition~\ref{P:Kob_dist_ub} applies to domains that are not covered by \cite{Forst_Ros_LocKobMet}.}
\smallskip

We now state the result from \cite{Bha_Comp_Geo} alluded to above. To state it, we need, given a bounded
convex domain $\OM\subset \C^n$ with $\smoo^1$-smooth boundary, the notion of a function that \emph{supports
$\OM$ from the outside}. Roughly speaking, such a function is a convex function
$\Phi: ( B^{(n-1)}(0,r_0), 0 ) \to ( [0,\infty), 0)$ such that, for each $\xi\in \bdy\OM$, there exists a unitary
change of coordinate $(\xiz_1,\dots, \xiz_n)\equiv (\xiz', \xiz_n)$ centred at $\xi$ so that 
$\{\xiz_n=0\} = T^{\C}_{\xi}(\bdy\OM)$ and such that a small open patch of $\bdy\OM$ around $\xi$
lies on the convex side of the surface $\{(\xiz', \xiz_n)\in B^{(n-1)}(0,r_0)\times \unitdisk \mid \iprt(\xiz_n)=\Phi(\xiz')\}$
(see \cite[Definition~1.5]{Bha_Comp_Geo}).
Now, for an arbitrary $\alpha>0$, let $\Psi_{\alpha}:[0,\infty) \to [0,\infty)$ be defined by
\[
  \Psi_{\alpha}(x) \defeq
  \begin{cases}
  \exp(-1/x^{\alpha}), &\text{if } x>0, \\
  0,                   &\text{if } x=0.
  \end{cases}   
\]
With these preparations, the result mentioned above is:

\begin{result}[Bharali {\cite[Theorem~1.4]{Bha_Comp_Geo}}] \label{R:Bha_Ext_Comp_Geo}
	Let $\OM \subset \C^n$ be a bounded convex domain with $\smoo^1$-smooth boundary. Suppose $\OM$ is supported
	from the outside by a function of the form $\Phi(z') \defeq \Psi_{\alpha}(\|z'\|)$, where $0 < \alpha <1$. Then every
	complex geodesic of $\OM$ extends continuously to $\overline{\unitdisk}$.	
\end{result}

The above result has recently been extended to certain convex domains with non-smooth boundaries; see
\cite[Theorem~1.7]{Bha_Zim_GoldDom}. The hypothesis of Result~\ref{R:Bha_Ext_Comp_Geo}
is such that it admits domains $\OM$
having boundary points that are not of finite type. As for the first four results in this section: their hypotheses \emph{manifestly}
cover the case where the domains involved have boundary points of infinite type. This is relevant because, by a result
of Zimmer \cite[Theorem~1.1]{Zim_GHyperbolicFinType}\,---\,given a bounded convex domain $\OM$ with
$\smoo^\infty$-smooth boundary and equipped with the Kobayashi distance $\koba_{\OM}$\,---\,if $\bdy\OM$ has
points of infinite type, then $(\OM, \koba_{\OM})$ is \emph{not} Gromov hyperbolic. Thus, \emph{not only} is
Theorem~\ref{T:contExtIsom} (as is Result~\ref{R:Zim_Cont_Ext_Isom} or Result~${\ref{R:Zim_Cont_Ext_Isom}}^{\!\prime}$)
a result involving domains with low boundary regularity, but it is one where $(\OM_1, \koba_{\OM_1})$,
$(\OM_2, \koba_{\OM_2})$ are not necessarily Gromov hyperbolic. I.e., a very natural condition under
which one may expect continuous extension to $\overline{\OM}_1$ of $\Omega_1 \to \Omega_2$
Kobayashi isometries is unavailable\,---\,and this work is an inquiry into what other kinds of hypotheses suffice.

\smallskip 

Result~\ref{R:Bha_Ext_Comp_Geo} and Result~\ref{R:Zim_Cont_Ext_Isom} both address the extension of complex geodesics
and have apparently similar hypotheses. But neither subsumes the other. Also note that in
Result~\ref{R:Bha_Ext_Comp_Geo} no constraints are placed on the way in which $\bdy \OM$ behaves in the
\emph{complex-normal} directions, but some degree of control is
required in the complex-tangential directions. This is in stark contrast to Result~\ref{R:Zim_Cont_Ext_Isom} (or
Result~${\ref{R:Zim_Cont_Ext_Isom}}^{\!\prime}$) and to our theorem. 
These together suggest the following

\begin{conjecture} \label{Cnj:Cplx_Geod}
	Let $\OM$ be a bounded convex domain that has $\smoo^1$-smooth boundary and is $\C$-strictly convex.
	Then every complex geodesic of
	$\OM$ extends continuously to $\overline{\unitdisk}$.
\end{conjecture}

\noindent{With the techniques currently known, this seems to be difficult to prove. Theorem~\ref{T:contExtIsom} may
be seen as evidence
in support of this conjecture.}

\smallskip

Before closing this section, we must mention a recent result in a similar vein by Bracci--Gaussier--Zimmer
\cite[Corollary~1.6]{Bra_Gau_Zim_HomeoExtQ-ISom}. This result concerns the continuous extension of
$\OM_1\to \OM_2$ Kobayashi quasi-isometries that are homeomorphisms. While this result involves no
assumption on the boundary regularity of $\OM_1$ or $\OM_2$, necessarily $\dim(\OM_1) = \dim(\OM_2)$.
Furthermore $(\OM_1, \koba_{\OM_1})$ is required to be Gromov hyperbolic. Thus, in view of our remarks
above, \cite[Corollary~1.6]{Bra_Gau_Zim_HomeoExtQ-ISom} is quite different from Theorem~\ref{T:contExtIsom}.

\smallskip 

The plan of this paper is as follows: in Section~\ref{S:Tech_prelim}, we collect some preliminary results that
are not immediately related to Theorem~\ref{T:contExtIsom} but which will play a crucial role in its proof. In
Section~\ref{S:Some_facts_convex_domains}, we collect three relevant facts about convex domains in $\C^n$.
In Section~\ref{S:Ess_props}, we prove the propositions that enable Result~\ref{R:Zim_Cont_Ext_Isom} to be generalized to
Theorem~\ref{T:contExtIsom}. The result of Zimmer that we generalize, which leads to Theorem~\ref{T:contExtIsom}, is
\cite[Proposition~4.3]{Zim_CharDomLimAut}: our generalization is Proposition~\ref{P:param_as_K-a-g}. Finally, in
Section~\ref{S:proof_main_thm}, we provide the proof of Theorem~\ref{T:contExtIsom}. In all these sections, $\|\bcdot\|$ will
denote the Euclidean norm. 

\medskip

\section{Technical preliminaries} \label{S:Tech_prelim}
In this section we present some results that play a supporting role in the proofs of the main results in Section~\ref{S:Ess_props} and,
therefore, of our main theorem. The first result, by S.E. Warschawski, is the principal tool that enables us to deal with the
low regularity of $\bdy\OM_1$ and $\bdy\OM_2$ in Theorem~\ref{T:contExtIsom}.

\smallskip

To state this result, we need to fix some terminology. Given a rectifiable arc $\Gamma$ in $\C$, we say that
$\Gamma$ has a \emph{a continuously turning tangent} if there is a $\smoo^1$-smooth diffeomorphism
$\gamma: I\to \Gamma$, where $I$ is an interval. Note, in particular, that $\gamma'$ is non-vanishing.
Given that $\Gamma$ has a continuously turning tangent, a \emph{tangent angle} at any point
$\zeta \in \Gamma$ refers to the smaller of the two angles determined by the intersection of
$T_{\zeta}\Gamma$ with a \textbf{fixed} line $\ell$ in $\C$. While different choices of $\ell$ define
different tangent-angle functions on $\Gamma$, the difference between the tangent angles\,---\,determined
by some fixed $\ell$\,---\,at two points $\zeta_1, \zeta_2\in \Gamma$ depends \textbf{only} on $\zeta_1$ 
and $\zeta_2$ (and, of course, on $\Gamma$), i.e., is independent of $\ell$. For this reason, in the
following result\,---\,and in all applications of it\,---\,we shall use the phrase ``the tangent angle'' without
any further comment. If $\Gamma$ is a closed rectifiable Jordan curve in $\C$, analogous observations
can be made about arc length. With these words, we can now state the following:

\begin{result}[{\cite[Theorem~1]{Warschawski_DiffBdyConfMap}}]\label{R:Warsch_ext_bihol}
	Let $C$ be a closed rectifiable Jordan curve in $\C$ and let $C$ have a continuously turning tangent in a
	$C$-open neighbourhood of a point $\zeta_0\in C$. Suppose that the tangent angle $\tau(s)$
	as a function of arc length
	$s$ has a modulus of continuity $\omega$ at the point $s_0$ corresponding to $\zeta_0$\,---\,i.e.,
	there exists a constant $\sigma > 0$ such that
	\begin{equation}
	  |\tau(s)-\tau(s_0)| \leqslant \omega(|s-s_0|) \text{ whenever } |s-s_0| \leqslant \sigma
	\end{equation}
	---\,that satisfies the following condition: 
	\begin{equation}\label{E:Dini_integral}
	  \int_0^{\sigma} \frac{\omega(t)}{t} dt < \infty.
	\end{equation}
	Let $f$ be a biholomorphic map of $\unitdisk$ onto the region $D$ enclosed by $C$ and let
	$\zeta_0 = f(z_0)$. Then
	\[
	  \lim_{\unitdisk\ni z \to z_0} \frac{f(z)-f(z_0)}{z-z_0} \defines f'(z_0)
	\]
	exists, and
	\[
	  \lim_{S\ni z \to z_0} f'(z) = f'(z_0)
	\]
	for any Stolz angle $S$ with vertex at $z_0$. Furthermore, $f'(z_0) \neq 0$. 
\end{result}

We refer the reader to \cite[Chapter~1]{Pommeremke_BdyBehvConfMap} for a definition
of a Stolz angle in $\unitdisk$.

\begin{remark}
Note that, since the $\omega$ appearing in the above result is a modulus of continuity,
it is a non-decreasing function on $[0,\sigma]$ (and is continuous at $0$). For this reason,
the integrand in \eqref{E:Dini_integral} is Lebesgue measurable. Secondly, in the statement
of Result~\ref{R:Warsch_ext_bihol}, we have tacitly used Carath{\'e}odory's
theorem to conclude that\,---\,given that $C$ is rectifiable\,---\,the map
$f: \unitdisk\to D$ in the above result extends to a homeomorphism of $\overline{\unitdisk}$.
\end{remark}

The following is an immediate corollary to the above result.

\begin{corollary} \label{C:coro_Warsch}
	In the set-up described in Result~\ref{R:Warsch_ext_bihol}, if $g$ denotes $f^{-1}$ then
	\begin{equation}
	L \defeq \lim_{D\ni \zeta \to \zeta_0} \frac{g(\zeta)-z_0}{\zeta-\zeta_0}
	\end{equation}
	exists and is non-zero.
\end{corollary}

We will also need the following simple lemma involving moduli of continuity.

\begin{lemma} \label{L:mod_of_cont_sub-add}
	Let $f$ be a real-valued function defined on a ball $B^{(n)}(a,r) \subset \C^n$. Then the modulus of continuity $\omega$ of $f$ on
	$B^{(n)}(a,r)$ is sub-additive, i.e., for all $s,t \in [0,2r)$ such that $s+t<2r$, $\omega(s+t) \leqslant \omega(s)+\omega(t)$.
\end{lemma}

\begin{proof}
	Suppose $x,y \in B^{(n)}(a,r)$ and $\|x-y\| \leqslant s+t$. If $\|x-y\| \leqslant s$, then $|f(x)-f(y)| \leqslant \omega(s) \leqslant\
	\omega(s)+\omega(t)$. Now suppose that $\|x-y\| > s$. Note that
	\[
	\left|f(x)-f\left(x+s\frac{y-x}{\|y-x\|}\right)\right| \leqslant \omega(s)
	\]
	and
	\begin{align*}
	\left|f\left(x+s\frac{y-x}{\|y-x\|}\right)-f(y)\right| &=
	\left|f\left(x+s\frac{y-x}{\|y-x\|}\right)-f\left(x+\|y-x\|\frac{y-x}{\|y-x\|}\right)\right| \\
	&\leqslant \omega(\|y-x\|-s) \leqslant \omega(t).
	\end{align*}
	Consequently, by the triangle inequality, $|f(y)-f(x)| \leqslant \omega(s)+\omega(t)$. Since $x$ and $y$ were arbitrary
	points in $B^{(n)}(a,r)$ satisfying $\|x-y\| \leqslant s+t$, it follows
	that $\omega(s+t) \leqslant \omega(s)+\omega(t)$.
\end{proof}

\smallskip

\section{Some facts about convex domains} \label{S:Some_facts_convex_domains}
In this section we record some facts about convex domains in $\C^n$. The first two were proved by Zimmer in \cite{Zim_CharDomLimAut}. All
of them are needed in the proof of Theorem~\ref{T:contExtIsom}. The first result concerns a lower bound for the Kobayashi distance on
arbitrary convex domains. First, some notation: in what follows, given a domain $\OM \subset \C^n$, $\koba_{\OM}$ will denote the
Kobayashi pseudodistance on $\OM$ and $\dkoba_{\OM}$ will denote the Kobayashi pseudometric on $\OM$.

\begin{result}[{\cite[Lemma~4.2]{Zim_CharDomLimAut}}] \label{R:lb_Kob_dist_hyp}
	Let $\OM \varsubsetneq \C^n$ be a convex domain and $H \subset \C^n$ be a complex affine hyperplane such that $\OM \cap H =
	\varnothing$. Then, for every $z_1,z_2 \in \OM$,
	\begin{equation}
	\koba_{\OM}(z_1,z_2) \geqslant \frac{1}{2} \left| \log\left( \frac{\distance(z_1,H)}{\distance(z_2,H)} \right) \right|.
	\end{equation} 
\end{result}

The next result is a sharper lower bound for the Kobayashi distance between a pair of points in a bounded convex domain with
$\smoo^1$-smooth boundary under an additional assumption. At this point, we wish to state a key clarification about our notation. 
Whenever $\OM \subset \C^n$ is a $\smoo^1$-smoothly bounded domain and $\xi \in \bdy \OM$, $(\xi+T^{\C}_{\xi}(\bdy \OM))$ will be
understood to be a certain \emph{set} in $\C^n$. $T_{\xi}(\bdy \OM)$ will denote the real tangent space to $\bdy \OM$ at $\xi$ viewed
\emph{extrinsically}: i.e., as a real hyperplane in $\R^{2n} \equiv \C^n$ taking into account that $\bdy \OM$ is $\smoo^1$-smoothly
embedded in $\C^n$. Then,
\[
T^{\C}_{\xi}(\bdy \OM) \defeq T_{\xi}(\bdy \OM) \cap i T_{\xi}(\bdy \OM),
\]
with $T_{\xi}(\bdy \OM)$ being viewed extrinsically.

\begin{result}[{\cite[Lemma~4.5]{Zim_CharDomLimAut}}] \label{R:lb_Kob_dist_diff_tanspcs}
	Let $\OM \subset \C^n$ be a bounded convex domain with $\smoo^1$-smooth boundary. Let $\xi,\xi' \in \bdy \OM$ and suppose that 
	$\xi+T^{\C}_{\xi}(\bdy \OM) \neq \xi'+T^{\C}_{\xi'}(\bdy \OM)$. Then there exist constants $\epsilon, C > 0$ such that for every $p
	\in \OM$ with $\distance(p,\xi+T^{\C}_{\xi}(\bdy \OM)) \leqslant \epsilon$ and every $q \in \OM$ with
	$\distance(q,\xi'+T^{\C}_{\xi'}(\bdy \OM)) \leqslant \epsilon$,
	\begin{equation}
	\koba_{\OM}(p,q) \geqslant \frac{1}{2}\log\left( \frac{1}{\dtb{\OM}(p)} \right) + \frac{1}{2} \log\left( \frac{1}{\dtb{\OM}(q)}
	\right) - C.
	\end{equation}
\end{result}

\noindent{Here, for any $z \in \OM$, $\dtb{\OM}(z) \defeq \distance(z,\OM^{\cmpl})$.}

\smallskip

The following result provides bounds for the Kobayashi \emph{metric} on convex domains.

\begin{result}[Graham {\cite[Theorem~3]{Graham90}}, also see {\cite{Graham91}}] \label{R:Grahams_result}
	Let $\OM \subset \C^n$ be a convex domain. Given $p \in \OM$ and $v \in T_p^{(1,0)}\OM$,
	we let $r_{\OM}(p,v)$ denote the supremum of the radii of the disks centred at $p$, tangent to $v$, and included in $\OM$. Then
	\begin{equation} \label{E:Grahams_bounds}
	\frac{\|v\|}{2 r_{\OM}(p,v)} \leqslant \dkoba_{\OM}(p,v) \leqslant \frac{\|v\|}{r_{\OM}(p,v)}.
	\end{equation} 	
\end{result}

\smallskip

\section{Essential Propositions} \label{S:Ess_props}

The goal of this section is to prove certain technical results that are essential for extending the scope of an idea
in \cite{Zim_CharDomLimAut} to the sorts of domains considered in Theorem~\ref{T:contExtIsom}. Specifically: that
inward-pointing normals can be parametrized as $K$-almost-geodesics for some $K \geqslant 1$. 
In \cite{Zim_CharDomLimAut}, this relies on a construction by Forstneric--Rosay in
\cite[Proposition~2.5]{Forst_Ros_LocKobMet} for estimating
effectively the Kobayashi distance close to the boundary of a domain $\OM$ whose
boundary is of class $\smoo^{1,\alpha}$.

\begin{definition}[Zimmer {\cite[Definition~3.2]{Zim_CharDomLimAut}}]
	Let $\OM \subset \C^n$ be a bounded domain. For $K \geqslant 1$, by a \emph{$K$-almost geodesic} in $\OM$
	(with respect to the Kobayashi distance) we mean a mapping $\sigma: I \to \OM$, where $I$ is an interval in $\R$,
	such that
	\begin{enumerate}
		\item $|s-t|-\log(K) \leqslant \koba_{\OM}(\sigma(s),\sigma(t)) \leqslant |s-t|+\log(K) \quad \forall \, s,t \in I$, and
		\item $\koba_{\OM}(\sigma(s),\sigma(t)) \leqslant K|s-t| \quad \forall \, s,t \in I$. 
	\end{enumerate}	
\end{definition}

The Forstneric--Rosay estimate involves embedding a certain compact planar set
$\overline{\mathcal{D}}$ with $0\in \bdy\mathcal{D}$ into $\OM$ so that its image osculates $\bdy\OM$
at the image of $0$. Since the domains considered in Theorem~\ref{T:contExtIsom} need not necessarily have
boundaries of class $\smoo^{1,\alpha}$, we must modify \emph{significantly} the constructions in
\cite[Proposition~2.5]{Forst_Ros_LocKobMet}, starting with a class of planar domains better adapted to
the domains $\OM_1$ and $\OM_2$ of Theorem~\ref{T:contExtIsom}.
\smallskip

Such a domain (which must contain $0$ in its boundary) must have a defining function that is $\smoo^1$ near $0$ whose derivative (while
not necessarily $\alpha$-H{\"o}lder-continuous for any $\alpha \in (0,1)$) will have a modulus of continuity that satisfies a Dini
condition. To this end, with $\omega$ as in Theorem~\ref{T:contExtIsom}, we define the function $h : (-2r, 2r)\to [0, \infty)$ as 
follows:
\[
  h(t) \defeq \begin{cases}
  			\,{\displaystyle \int_t^0} \omega(-y) dy, &\text{if } t<0, \\
  			{} & {} \\
  			\,{\displaystyle \int_0^t} \omega(y)  dy, &\text{if } t\geqslant 0.
  \end{cases}
\]
The following properties of $h$ are easily verified: $h(0)=0$; $h$ is strictly increasing on $[0,2r)$, and strictly
decreasing on $(-2r,0]$; and $h'(0)=0$. Then, for $\alpha,\tau>0$, consider the domain
\[
\geodModDomB \defeq \{\zeta = s+it \in \C \mid |t| < \tau, \ \alpha h(t) < s < \tau\}.
\]
The following property of the domains $\geodModDomB$ is obvious from the definition: if $\zeta = s+it \in \geodModDomB$,
then $|t| \leqslant h^{-1}(s/\alpha)$. Near $0$, a defining function for $\geodModDomB$ is
$\varrho(s,t) \defeq \alpha h(t) -s$. Its total derivative at the point $(s,t)$, 
$(D\varrho)(s,t)$, with respect to the standard basis of $\R^2$, is
\[
\big[\!-\!1 \;\, \alpha h'(t) \big].
\]
It is easily checked that the modulus of continuity of $D\varrho$ at $0$ is $\alpha \omega$: i.e., for
every $t \in (-2r,2r)$, $\|(D\varrho)(s,t)-(D\varrho)(0)\| = \alpha \omega(|t|)$. 

\smallskip

We will use the following fact in our proof below: if $w \in \C^n$, then $\multByi(w)$ is orthogonal to $w$ with respect
to the standard real inner product on $\R^{2n} \longleftrightarrow \C^n$. With this remark, we now state
and prove the following proposition.

\begin{proposition} \label{P:emb_mod_dom}
	Let $\Omega \subset \C^n$ be a bounded convex domain having the properties common to
	$\OM_1$ and $\OM_2$ as stated in Theorem~\ref{T:contExtIsom}.
	For
	$\xi \in \bdy \Omega$, let $\Psi_{\xi} : \C \to \C^n$ denote the $\C$-affine map
	\[
	\Psi_{\xi}(\zeta) \defeq \xi + \zeta \eta^{\xi} \quad \forall \, \zeta \in \C.
	\]
	Then there exist constants $\alpha,\tau > 0$ such that, for every $\xi \in \bdy \Omega$, $\Psi_{\xi}(\geodModDomB)
	\subset \Omega$. 
\end{proposition}

\begin{proof}
We are given a $\smoo^1$ defining function $\rho$ defined on a neighbourhood $U$ of $\bdy \Omega$ and we are given an $r>0$
such that, for every $\xi \in \bdy \Omega$, the directional derivative $D_{\multByi\left( \eta^{\xi} \right)}\rho$ has on
$B^{(n)}(\xi,r)$ modulus of continuity $\omega$. 
We shall identify $\mathscr{L}(\R^{2n},\R)$ with $\R^{2n}$ via the matrix representation of the elements of
$\mathscr{L}(\R^{2n},\R)$ relative to the standard basis of $\R^{2n}$.
Since $D\rho$ does not vanish on $\bdy \Omega$, there is an $m>0$ such that, for every $\xi \in \bdy \Omega$,
$\|(D\rho)(\xi)\| \geqslant m$. Furthermore, if we choose a neighbourhood $V$ of $\bdy \Omega$ in $U$ such that
$\overline{V}$ is a compact subset of $U$, then $D\rho$ is uniformly continuous on $V$. In particular, there is a
$\delta_0$, $0<\delta_0<r$, such that
\begin{equation} \label{E:D_rho_diff}
  \|(D\rho)(\xi)-(D\rho)(\xi')\| \leqslant m/2 \quad \forall \, \xi,\xi' \in V \text{ such that } \|\xi-\xi'\| \leqslant
  \delta_0. 
\end{equation}
Choose $\tau > 0$ so small that $\sqrt{2}\tau < \delta_0$; it then follows that 
\[
  \{s+it \in \C \mid |t|<\tau, \ 0<s<\tau\} \subset D(0,\delta_0).
\]
Then, for any $\alpha > 0$, $\geodModDomB \subset D(0,\delta_0)$.
We fix a value of $\alpha \geqslant 1$ so large that $2/\alpha \leqslant m/4$. We shall
soon see the reason for this choice. We may also need to shrink $\tau$ further.
The precise value of $\tau$ that works will be presented below. 

For the rest of the proof we fix
$\xi \in \bdy \Omega$ and $\zeta \in \geodModDomB$. In what follows, $a+ib$ 
($a, b\in \R^n$) will, for simplicity of notation, denote either a complex vector or the vector
$(a_1, b_1,\dots, a_n, b_n)\in \R^{2n}$\,---\,the intended meaning being clear from the context. 
By Taylor's theorem, and writing $\zeta = s+it$ (and $\langle\bcdot\,, \bcdot\rangle$ denoting the
usual inner product on $\R^{2n}$)
\begin{align}
  \rho(\xi + \zeta \eta^{\xi}) =&\ \rho(\xi) + (D\rho)(\xi)(\zeta\eta^{\xi})
    + \int_0^1 \big((D\rho)(\xi+x\zeta\eta^{\xi})-(D\rho)(\xi) \big)(\zeta\eta^{\xi}) dx \notag \\
  =&\ s\langle \nabla\rho(\xi), \eta^{\xi}\rangle + t \langle\nabla\rho(\xi), \multByi(\eta^{\xi})\rangle
      + s \int_0^1 \big((D\rho)(\xi+x\zeta\eta^{\xi})-(D\rho)(\xi) \big)(\eta^{\xi}) dx \notag \\
      &+ t \int_0^1 \big( (D\rho)(\xi+x\zeta\eta^{\xi})-(D\rho)(\xi) \big)(\multByi(\eta^{\xi})) dx \quad
      \text{[since $\rho(\xi)=0$]} \notag \\
  \leqslant&\ -sm + s \int_0^1 \big| ((D\rho)(\xi+x\zeta\eta^{\xi}) - (D\rho)(\xi))(\eta^{\xi}) \big| dx \notag \\
    &+ |t| \int_0^1 \big| \big( D_{\multByi(\eta^{\xi})}\rho \big)(\xi+x\zeta\eta^{\xi}) - \big( D_{\multByi(\eta^{\xi})}\rho
    \big)(\xi) \big| dx  \quad 
    \text{[since $\multByi(\eta^{\xi}) \perp \nabla\rho(\xi)$].}    \label{E:xprsn_rho_two_intgrls}                  
\end{align}
Since, for every $\xi \in \bdy \Omega$, every $x \in [0,1]$ and every $\zeta \in \geodModDomB$,
$\|(\xi+x\zeta\eta^{\xi})-\xi\| = x |\zeta| < \delta_0$,
\[
  \|(D\rho)(\xi+x\zeta\eta^{\xi})-(D\rho)(\xi)\| \leqslant \frac{m}{2},
\]
by \eqref{E:D_rho_diff}.
Therefore the second term on the right hand side of \eqref{E:xprsn_rho_two_intgrls} is less than or equal to $sm/2$.
As for the third term, note that for every $x \in [0,1]$,
\[
  \big| \big( D_{\multByi\left( \eta^{\xi} \right)}\rho \big)(\xi+x\zeta\eta^{\xi}) - \big( D_{\multByi\left( \eta^{\xi}
  \right)}\rho \big)(\xi) \big| \leqslant \omega(\|(\xi+x\zeta\eta^{\xi})-\xi\|) = \omega(x |\zeta|).
\]
So the third term on the right hand side of \eqref{E:xprsn_rho_two_intgrls} is less than or equal to
\[
  |t| \int_0^1 \omega(x |\zeta|) dx.
\] 
Therefore we get, from \eqref{E:xprsn_rho_two_intgrls},
\begin{equation} \label{E:xprsn_rho_only_omega}
  \rho(\xi + \zeta\eta^{\xi}) \leqslant -sm + (sm/2) + |t| \int_0^1 \omega(x |\zeta|) dx.
\end{equation}
Since $\zeta \in \geodModDomB$,
\begin{equation} \label{E:xprsn_mod_lambda}
  |\zeta|
  \leqslant \big( s^2 +  (h^{-1}(s/\alpha))^2 \big)^{1/2}
  = h^{-1}( s/\alpha ) \left( 1 + \left( \frac{s}{h^{-1}(s/\alpha)} \right)^{\!\!2} \right)^{1/2}.
\end{equation}
Since $h'(0) = 0$, we have $\lim_{x \to 0+} x/h^{-1}(x) =0$. Therefore, we can
shrink $\tau$ so that, 
\begin{equation} \label{E:condition_s_h-inverse_s}
  x/h^{-1}(x) \leqslant 1/\alpha \quad \forall x\in  (0,\tau).
\end{equation}
Now, from \eqref{E:xprsn_mod_lambda}, the fact that
$0 < s/\alpha < \tau$ (since $s/\alpha \leqslant s$ by our choice of $\alpha$), and
\eqref{E:condition_s_h-inverse_s}, we have:
\[ 
  1 + \left( \frac{s}{h^{-1}(s/\alpha)} \right)^{\!\!2} =
  1 + \alpha^2 \left( \frac{s/\alpha}{h^{-1}(s/\alpha)} \right)^{\!\!2} \leqslant 2.
\]
From the last inequality and \eqref{E:xprsn_mod_lambda},
\[
  |\zeta| \leqslant \sqrt{2}h^{-1}(s/\alpha).
\]
Using the above in \eqref{E:xprsn_rho_only_omega} we get that
\begin{align*}
  \rho(\xi + \zeta \eta^{\xi}) &\leqslant -(sm/2) +
  |t| \int_0^1 \omega\left( \sqrt{2} x h^{-1}(s/\alpha) \right) dx \notag \\
  &\leqslant -(sm/2) + h^{-1}(s/\alpha) \int_0^1 \omega\left( 2 x h^{-1}(s/\alpha) \right) dx \notag \\
  &\leqslant s \left( -\frac{m}{2} + \frac{2h^{-1}(s/\alpha)}{s} \int_0^1 \omega\left( x h^{-1}(s/\alpha) \right) dx
  \right) \notag \\
  &= s \left( -\frac{m}{2} + \frac{2}{s} \int_0^{h^{-1}(s/\alpha)} \omega(u) du\right)
       \quad \text{[by change of variables]} \notag \\
  &= s \left( -\frac{m}{2} + \frac{2}{\alpha} \right) \leqslant -\frac{sm}{4} < 0,
\end{align*}
by the choice of $\alpha$ discussed above. We note here that the third inequality follows
from Lemma~\ref{L:mod_of_cont_sub-add}.
Therefore,
$\xi+\zeta \eta^{\xi} \in \Omega$. Since $\xi \in \bdy \OM$ and $\zeta \in \geodModDomB$ were arbitrary, the proof is complete. 
\end{proof}

The proof of Theorem~\ref{T:contExtIsom}\,---\,as we shall see\,---\,relies crucially on the conclusion
of Result~\ref{R:Warsch_ext_bihol}, for the point $z_0=1$, when applied to the domains $\geodModDomB$.
We must therefore verify that the hypotheses of that result hold for $\geodModDomB$. 
It is enough to show that the modulus of continuity of the tangent angle to $\bdy \geodModDomB$ near $0$, regarded as a
function of arc length, satisfies a Dini condition. Before we do this, we note the following elementary fact:
\begin{equation} \label{E:tan_cond}
|\tan^{-1}(x)| \leqslant |x| \quad \forall \, x \in \R.
\end{equation}
We also note that given $\alpha$ and $\tau$, a parametrization of $\bdy \geodModDomB$ near $0$ is given by
$\Phi \defeq y \mapsto (\alpha h(y), y) : [-\epsilon_0,\epsilon_0] \to \R^2$, where $\epsilon_0$ is a suitably small
positive quantity depending on $\alpha$ and $\tau$. Therefore the tangent angle to $\bdy \geodModDomB$
near $0$, as a function of $y$, is
\begin{equation} \label{E:tangent_angle_terms_of_y}
\widehat{\theta}(y) = \tan^{-1}(\alpha h'(y)) \quad \forall \, y \in [-\epsilon_0,\epsilon_0].
\end{equation}
(In this instance, the line $\ell$, as introduced in the explanations preceding Result~\ref{R:Warsch_ext_bihol},
is the imaginary axis of $\C$.)
Now we present the following lemma.

\begin{lemma} \label{L:mod_dom_sat_Warsch_con}
	The tangent angle $\theta$ of $\bdy \geodModDomB$ near $0$, regarded as a function of arc length, has a modulus of continuity that
	is dominated by $\alpha \omega$ (and therefore satisfies a Dini condition).	
\end{lemma}

\begin{proof}
First we determine the arc length as a function of $y$. We will reckon the (signed) arc length $s$ from $0$ and
such that $s(x+iy) < 0$ for $x+iy\in \bdy \geodModDomB$ and $y<0$, and 
$s(x+iy) > 0$ for $x+iy\in \bdy \geodModDomB$ and $y>0$
(we are only interested in the arc length near $0$). Using the parametrization $\Phi$ referred to just prior to
\eqref{E:tangent_angle_terms_of_y}, we see that the function that gives the arc length as a function of $y$,
which we denote by $G$, is
\begin{equation} \label{E:arclength_terms_of_y} 
G(y) = \int_0^y \|\Phi'(t)\| dt = \int_0^y \big[ 1+\alpha^2\omega(|t|)^2  \big]^{1/2} dt
\end{equation}
for all $y \in (-\epsilon_0,\epsilon_0)$. 
Clearly,
\begin{equation}\label{E:G_majorant}
  |G(y)| \geqslant |y| \quad \forall\,y \in (-\epsilon_0,\epsilon_0).
\end{equation}
Note that $G$ is a strictly increasing odd function on $(-\epsilon_0,\epsilon_0)$.
So $G^{-1}$ is a function that is defined on $(-G(\epsilon_0),G(\epsilon_0))$ and is strictly increasing.
Taking $y = G^{-1}(s)$, $s\in (-G(\epsilon_0),G(\epsilon_0))$, in \eqref{E:G_majorant}, we
get
\begin{equation} \label{E:G_inverse_majorant} 
|G^{-1}(s)| \leqslant |s| \quad \forall\,s \in (-G(\epsilon_0),G(\epsilon_0)).
\end{equation}
Now the function $\theta$ that gives the tangent angle as a function of arc length is
\[
\theta(s) = \widehat{\theta}(G^{-1}(s)) \quad \forall \, s \in (-G(\epsilon_0),G(\epsilon_0)).
\]
Recall that $|h'(y)| = \omega(|y|)$, and $\omega$ is continuous at $0$. Thus, we may suppose that $\epsilon_0$ is so small that, for every
$y \in (-\epsilon_0,\epsilon_0)$, $\alpha\omega(|y|) \leqslant 1$. Therefore, for an arbitrary $s \in (-G(\epsilon_0),G(\epsilon_0))$,
\begin{align*}
|\theta(s)| = |\widehat{\theta}(G^{-1}(s))|
		  &= \left| \tan^{-1}\big( \alpha h'(G^{-1}(s)) \big) \right| &&\text{[by \eqref{E:tangent_angle_terms_of_y}]} \\
		  &\leqslant \alpha |h'(G^{-1}(s))| && \text{[by \eqref{E:tan_cond}]} \\
		  &\leqslant \alpha \omega(|s|) && \text{[by \eqref{E:G_inverse_majorant}]}.
\end{align*}
This gives us the required result. 	
\end{proof}

\begin{remark}
The significance of Lemma~\ref{L:mod_dom_sat_Warsch_con} is as follows:
for every $\alpha>0, \tau>0$, the domain $\geodModDomB$ satisfies the hypotheses of
Result~\ref{R:Warsch_ext_bihol} at $0\in \bdy\geodModDomB$. Thus,
Corollary~\ref{C:coro_Warsch} holds.
\end{remark}

We are now ready to state and prove a generalization of Proposition~4.3 in \cite{Zim_CharDomLimAut}. The
generalization of the latter result
alone suffices to yield a generalization of Theorem~2.11 in \cite{Zim_CharDomLimAut}, which is fundamental
to establishing an extension-of-isometries theorem.

\begin{proposition} \label{P:param_as_K-a-g}
	Let $\Omega$ be an open convex subset of $\C^n$ having the properties
	possessed in common by $\OM_1$ and $\OM_2$ in the
	statement of Theorem~\ref{T:contExtIsom}.
	Then there exist $K,\epsilon > 0$ such that for every $\xi \in \bdy \Omega$,
	\begin{equation}
	\sigma_{\xi} \defeq t \mapsto \xi + \epsilon e^{-2t} \eta^{\xi} : [0,\infty) \to \Omega
	\end{equation}
	is a $K$-almost-geodesic. 
\end{proposition}

\begin{proof}
Our proof will resemble, in \emph{essence}, the proof of Proposition~4.3 in \cite{Zim_CharDomLimAut}.
The two proofs will differ in the key detail that we must work with the domains $\geodModDomB$, which are
adapted to the domain $\OM$ under consideration.

\smallskip	

By Proposition~\ref{P:emb_mod_dom} there exist $\alpha,\tau > 0$ such that for every $\xi \in \bdy \Omega$, $\xi +
\geodModDomB \eta^{\xi} \subset \Omega$.  
As $\geodModDomB$ is a bounded open convex subset of $\C$ symmetric
about the real axis, there exists a biholomorphism $g: \geodModDomB \to \unitdisk$ such that
$g \big( \geodModDomB \cap\R \big) = \unitdisk \cap \R$. By Carath\'{e}odory's theorem, $g$ extends to a
homeomorphism from
$\overline{\geodModDomB}$ to $\overline{\unitdisk}$. We may suppose, without loss of generality, that $g(0)=1$.
By the remark following the proof of Lemma~\ref{L:mod_dom_sat_Warsch_con}, we see that we can apply 
Corollary~\ref{C:coro_Warsch} to $g$ to conclude that
\[
  \lim_{\geodModDomB\ni z \to 0} \frac{g(z)-g(0)}{z} = \lim_{\geodModDomB\ni z \to 0}
  \frac{g(z)-1}{z} 
\]
exists (call it $k$) and is non-zero. Therefore $k$ is a negative real number. Thus, there exist constants
$\epsilon > 0$ and $\kappa \geqslant 1$ such that $t \in \geodModDomB$ whenever 
$0 < t \leqslant \epsilon$ and
\begin{equation} \label{E:Phi_near_0}
  0 \leqslant 1-\kappa{}t \leqslant g(t) \leqslant 1-\kappa^{-1}t \quad
  \forall\,t: 0 < t \leqslant \epsilon.
\end{equation}

Then for $t_1,t_2$ such that $0 < t_1 < t_2 \leqslant \epsilon$, we have
\begin{align*}
  \koba_{\geodModDomB}(t_1,t_2) &= \koba_{\unitdisk}(g(t_1), g(t_2))
  		= \frac{1}{2}\log\left( \frac{\big( 1+g(t_1) \big)
  		   \big( 1-g(t_2) \big)}{\big( 1+g(t_2) \big)\big( 1-g(t_1) \big)} \right) \\
  &\leqslant \frac{1}{2}\log2 + \frac{1}{2}\log\left(\frac{1-g(t_2)}{1-g(t_1)}\right)\\
  &\leqslant \frac{1}{2}\log2 + \log(\kappa) + \frac{1}{2}\log(t_2/t_1),  
\end{align*}
by \eqref{E:Phi_near_0}. So for $\xi \in \bdy \Omega$ and $t,s \in [0,\infty)$ arbitrary,
\begin{align*}
  \koba_{\Omega}(\sigma_{\xi}(t),\sigma_{\xi}(s)) &= \koba_{\Omega}
  \left( \Psi_{\xi}\big(\epsilon e^{-2t}\big),\Psi_{\xi}\big(\epsilon
  e^{-2s}\big) \right) \notag \\
  &\leqslant \koba_{\geodModDomB}(\epsilon e^{-2t},\epsilon e^{-2s}) \notag \\
  &\leqslant \log(\sqrt{2}\kappa) + (1/2) \log(\epsilon e^{-2t}/\epsilon e^{-2s}) \notag \\
  &=\log(\sqrt{2}\kappa) + (s-t), 
\end{align*}
provided $s\geqslant{}t$, where $\Psi_{\xi}$ is as introduced in Proposition~\ref{P:emb_mod_dom}. In general
\begin{equation} \label{E:upper_bd_sigma_xi}
  \koba_{\Omega}(\sigma_{\xi}(t),\sigma_{\xi}(s)) \leqslant \log(\sqrt{2}\kappa) + |s-t| \quad \forall \, s,t \in [0,\infty).
\end{equation}
Consider the complex affine hyperplane $\xi+T^{\C}_{\xi}(\bdy \Omega)$ tangent to $\bdy \Omega$ at $\xi$. Of course,
$\xi+T^{\C}_{\xi}(\bdy \Omega)$ is a complex affine supporting hyperplane for $\Omega$ at $\xi$. For $t
\in \R$ arbitrary, the distance of $\sigma_{\xi}(t)$ from $\xi+T^{\C}_{\xi}(\bdy \Omega)$ is clearly $\epsilon e^{-2t}$. Consequently, by
Result~\ref{R:lb_Kob_dist_hyp},
\begin{equation} \label{E:lower_bd_sigma_xi}
  \koba_{\Omega}(\sigma_{\xi}(t),\sigma_{\xi}(s)) \geqslant (1/2) | \log( \epsilon e^{-2t}/ \epsilon e^{-2s} ) | = |s-t| \quad \forall \,
  s,t \in [0,\infty). 
\end{equation}
By \eqref{E:upper_bd_sigma_xi} and \eqref{E:lower_bd_sigma_xi}, each $\sigma_{\xi}$ is a $(1,\log(\sqrt{2}\kappa))$-quasi-geodesic.
It only remains to prove the Lipschitz nature of $\sigma_{\xi}$.

\smallskip

By the fact that the boundary of $\OM$ is $\smoo^1$, we
can, by shrinking $\epsilon$ if necessary, ensure that for every $\xi \in \bdy \Omega$, $\xi+\mathcal{B}_{\epsilon}\eta^{\xi}
\subset \Omega$,
where
\[
  \mathcal{B}_{\epsilon} \defeq \{ \zeta \in \C \mid 0 < \rprt(\zeta) < 2\epsilon, |\iprt(\zeta)| < \rprt(\zeta) \}.
\]
Elementary two-dimensional geometry then shows that there is a $C>0$ such that for every $\xi \in \bdy
\Omega$ and every $t \in [0,\infty)$,
\[ 
  r_{\Omega}(\sigma_{\xi}(t),\sigma'_{\xi}(t)) \geqslant 
  C \epsilon e^{-2t}.
\]
(In fact, given that $\xi+\mathcal{B}_{\epsilon}\eta^{\xi}\subset \Omega$, $C = 1/\sqrt{2}$ would work.)
Therefore,
by Graham's estimate\,---\,i.e., Result~\ref{R:Grahams_result}\,---\,for
every $\xi \in \bdy \Omega$ and every $t \in [0,\infty)$,
\[
  \kappa_{\Omega}(\sigma_{\xi}(t),\sigma'_{\xi}(t)) \leqslant \frac{\| \sigma'_{\xi}(t) \|}{ r_{\Omega}(\sigma_{\xi}(t),\sigma'_{\xi}(t)) }
  \leqslant \frac{2 \epsilon e^{-2t}}{C \epsilon e^{-2t}} = \frac{2}{C}.
\]
Consequently, for every $\xi \in \bdy \Omega$ and every $s,t \in [0,\infty)$,
\begin{equation} \label{E:sigma_xi_Lip_cond}
  \koba_{\Omega}(\sigma_{\xi}(s),\sigma_{\xi}(t)) \leqslant (2/C)|s-t|.
\end{equation}
Therefore, by \eqref{E:upper_bd_sigma_xi}, \eqref{E:lower_bd_sigma_xi} and \eqref{E:sigma_xi_Lip_cond}, it follows that for
every $\xi \in \bdy \Omega$, $\sigma_{\xi}$ is a $K$-almost-geodesic, where $K \defeq \max\{ \sqrt{2}\kappa,2/C \}$. 
\end{proof}

An outcome of one half of our argument for Proposition~\ref{P:param_as_K-a-g} is the following

\begin{proposition}\label{P:Kob_dist_gen_ub}
	Let $\Omega \subset \C^n$ be a bounded convex domain having the properties common
	to $\OM_1$ and $\OM_2$ as in the statement of Theorem~\ref{T:contExtIsom}. Let $z_0\in \OM$. Then, there
	exists a constant $C>0$ such that
	\[
	  \koba_{\OM}(z_0, z) \leqslant C + \frac{1}{2}\log\left( \frac{1}{\distance(z,\OM^{\cmpl})} \right)
	  \quad \forall z\in \OM.
	\]
\end{proposition}

\begin{remark}
	Since the domain $\OM$ in the statement of Proposition~\ref{P:Kob_dist_ub} is bounded, whence
	$\bdy\OM$ is compact, it is easy to see that Proposition~\ref{P:Kob_dist_ub} is a special case of the above.
\end{remark}

\begin{proof}
We abbreviate $\distance(z,\OM^{\cmpl})$ to $\dtb{\OM}(z)$. From the argument leading up to
\eqref{E:upper_bd_sigma_xi} in the proof of Proposition~\ref{P:param_as_K-a-g}, we conclude 
that there exist constants $\epsilon>0$ and $K \geqslant 1$ such that for every $\xi \in \bdy \OM$, and
$\sigma_{\xi}$ as in that proposition,
\begin{equation}\label{E:Kob_dist_sigma_xi}
  \koba_{\OM}(\sigma_{\xi}(t),\sigma_{\xi}(s)) \leqslant |s-t|+\log K \quad \forall s,t \in [0,\infty).
\end{equation}
By compactness, there exists a $\delta\in (0, \epsilon)$ such that
$\distance(\{ \sigma_{\xi}(0) \mid \xi \in \bdy \OM \}, \OM^{\cmpl}) \geqslant \delta$. 
It suffices to show that there exists $C > 0$ such that 
for every $z \in \OM$ with $\dtb{\OM}(z)<\delta$, $\koba_{\OM}(z_0,z) \leqslant C + 2^{-1} \log(1/\dtb{\OM}(z))$. 
Let $C' \defeq \sup_{\xi \in \bdy \OM} \koba_{\OM}(z_0,\sigma_{\xi}(0))$. So let
$z \in \OM$ with $\dtb{\OM}(z) < \delta$. Now fix a $\xi \in \bdy \OM$ such that $\|z-\xi\|=\dtb{\OM}(z)$. Clearly, 
then, there exists $t(z) \in (0,\infty)$ such that $z=\sigma_{\xi}(t(z))$.	
So $\koba_{\OM}(z_0,z) \leqslant \koba_{\OM}(z_0,\sigma_{\xi}(0)) +
\koba_{\OM}\big( \sigma_{\xi}(0),\sigma_{\xi}(t(z)) \big) \leqslant	
C'+t(z)+\log K$, by \eqref{E:Kob_dist_sigma_xi}. Since, by definition of
$\sigma_{\xi}$, $\dtb{\OM}(z) = \epsilon e^{-2t(z)}$, a
simple calculation shows that $t(z) \leqslant 2^{-1} \log(1/\dtb{\OM}(z))$. Therefore	
$\koba_{\OM}(z_0,z) \leqslant C + 2^{-1} \log(1/\dtb{\OM}(z))$, where $C \defeq C' + \log K$. This gives the desired conclusion.  
\end{proof}

\section{The proof of Theorem~\ref{T:contExtIsom}} \label{S:proof_main_thm}

The proof of Theorem~\ref{T:contExtIsom} requires the following conclusions: if $\OM \subset \C^n$ is a domain that has
the properties possessed in common by $\OM_1$ and $\OM_2$ in the statement of Theorem~\ref{T:contExtIsom} then
(we remind the reader that for $\xi \in \bdy \OM$, the set $(\xi+T^{\C}_{\xi}(\bdy \OM))$ is as described in
Section~\ref{S:Some_facts_convex_domains}):
\begin{enumerate}
	\item If $\xi \in \bdy \OM$ and $(p_{\nu})_{\nu \geqslant 1}$, $(q_{\mu})_{\mu \geqslant 1}$ are sequences in $\OM$ converging to
	$\xi$, then
	\[
	  \lim_{\nu,\mu \to \infty} (p_{\nu}|q_{\mu})_o = \infty.
	\] 
	\item If $\xi,\xi' \in \bdy \OM$ and $(p_{\nu})_{\nu \geqslant 1}, (q_{\mu})_{\mu \geqslant 1}$ are sequences in $\OM$ converging to
	$\xi$ and $\xi'$ respectively such that
	\[
	  \limsup_{\nu,\mu \to \infty} \; (p_{\nu}|q_{\mu})_o = \infty,
	\]
	then $\xi+T^{\C}_{\xi}(\bdy \OM) = \xi'+T^{\C}_{\xi'}(\bdy \OM)$. 
\end{enumerate}
In the above, $(\cdot\mid\cdot)_o$ denotes the Gromov product relative to the Kobayashi distance on $\OM$ and with respect to an
arbitrary but fixed base point $o \in \OM$. It is defined as 
\[
  (x|y)_o \defeq \frac{1}{2} \big( \koba_{\OM}(x,o)+\koba_{\OM}(y,o)-\koba_{\OM}(x,y) \big). 
\]

\smallskip

The above conclusions have been demonstrated by Zimmer under the conditions he states in \cite[Theorem~4.1]{Zim_CharDomLimAut}. We 
observe that what has actually been established in \cite[Theorem~4.1]{Zim_CharDomLimAut} is the following:

\begin{proposition} \label{P:nice_doms_Grom_hyp_app}
	Suppose $\OM$ is a bounded open convex subset of $\C^n$ having $\smoo^1$-smooth boundary. Suppose $\OM$ possesses the property that
	there exist constants $\epsilon > 0, K \geqslant 1$ such that, for each $\xi \in \bdy \OM$, the path
	\[
	  \sigma_{\xi} \defeq t \mapsto \xi + \epsilon e^{-2t} \eta^{\xi} : [0,\infty) \to \OM
	\]
	is a $K$-almost-geodesic. Then: 
	\begin{enumerate}
		\item \label{Cs:Grom_hyp_same_pt} If $\xi \in \bdy \OM$ and $(p_{\nu})_{\nu \geqslant 1}$, $(q_{\mu})_{\mu \geqslant 1}$ are
		sequences in $\OM$ converging to $\xi$, then
		\[
		  \lim_{\nu,\mu \to \infty} (p_{\nu}|q_{\mu})_o = \infty.
		\] 
		\item \label{Cs:Grom_hyp_diff_pts} If $\xi,\xi' \in \bdy \OM$ and $(p_{\nu})_{\nu \geqslant 1}, (q_{\mu})_{\mu \geqslant 1}$ are
		sequences in $\OM$ converging to $\xi$ and $\xi'$ respectively such that
		\[
		  \limsup_{\nu,\mu \to \infty} \; (p_{\nu}|q_{\mu})_o = \infty,
		\]
		then $\xi+T^{\C}_{\xi}(\bdy \OM) = \xi'+T^{\C}_{\xi'}(\bdy \OM)$. 
	\end{enumerate}
\end{proposition}

The condition on $\bdy \OM$ in \cite[Theorem~4.1]{Zim_CharDomLimAut} was required to obtain the property concerning the paths $\{
\sigma_{\xi} \mid \xi \in \bdy \OM \}$ stated in Proposition~\ref{P:nice_doms_Grom_hyp_app}. Other than this, there is absolutely no
difference between the proofs of \cite[Theorem~4.1]{Zim_CharDomLimAut} and Proposition~\ref{P:nice_doms_Grom_hyp_app}. We therefore omit
the proof of the latter.

\smallskip

Finally, we give the proof of Theorem~\ref{T:contExtIsom}.

\begin{proof}
First, we show that whenever $\xi \in \bdy \OM_1$, $\lim_{z \to \xi} F(z)$ exists. Since $F$ is an isometry with respect to the Kobayashi
distances, we see, from the definition of the Gromov product above, that, for every $z,w,o \in \OM_1$,
\[
  (z|w)_o = (F(z)|F(w))_{F(o)}.
\]
First note that if $\xi \in \bdy \OM_1$ and $(z_{\nu})_{\nu \geqslant 1}$ is a sequence in $\OM_1$ converging to $\xi$ such that
$(F(z_{\nu}))_{\nu \geqslant 1}$ converges to some point $\zeta \in \overline{\OM}_2$, then $\zeta \in \bdy\OM_2$. The reason is that, if
we fix a point $o \in \OM_1$ arbitrarily, we see that
\[
\lim_{\nu \to \infty} \koba_{\OM_2}(F(z_{\nu}),F(o)) = \lim_{\nu \to \infty} \koba_{\OM_1}(z_{\nu},o) = \infty,
\]  
by Lemma~\ref{R:lb_Kob_dist_hyp}. Consequently, $\zeta = \lim_{\nu \to \infty} F(z_{\nu})$ must belong to $\bdy \OM_2$. Thus, if $\xi \in
\bdy \OM_1$ and $(z_{\nu})_{\nu \geqslant 1}, (w_{\nu})_{\nu \geqslant 1}$ are sequences in $\OM_1$ converging to $\xi$
such that $(F(z_{\nu}))_{\nu \geqslant 1}$ and $(F(w_{\nu}))_{\nu \geqslant 1}$ converge to $\zeta_1, \zeta_2 \in \overline{\OM}_2$,
respectively, then $\zeta_1, \zeta_2 \in \bdy \OM_2$. Moreover,
\[
  \lim_{\nu \to \infty} (F(z_{\nu})|F(w_{\nu}))_{F(o)} = \lim_{\nu \to \infty} (z_{\nu}|w_{\nu})_o = \infty,
\]
by (\ref{Cs:Grom_hyp_same_pt}) of Proposition~\ref{P:nice_doms_Grom_hyp_app} above. Consequently, by (\ref{Cs:Grom_hyp_diff_pts}) of the
same proposition, $\zeta_1+T^{\C}_{\zeta_1}(\bdy \OM_2) = \zeta_2+T^{\C}_{\zeta_2}(\bdy \OM_2)$. Therefore, since $\OM_2$ is
$\C$-strictly convex, one has $\zeta_1= \zeta_2$. Since $\OM_2$ is bounded, so that any sequence in it has a convergent subsequence, the
above shows that $\lim_{z \to \xi} F(z)$ exists.

\smallskip

Then define $\widetilde{F}: \overline{\OM}_1 \to \overline{\OM}_2$ by letting $\widetilde{F}$ equal $F$ on $\OM_1$ and by letting 
$\widetilde{F}(\xi)$, for $\xi \in \bdy \OM_1$, be $\lim_{z \to \xi}F(z)$. It is routine to show that $\widetilde{F}$ is continuous, in
view of the conclusions of the previous paragraph. This completes the proof.	
\end{proof}

\smallskip

\section*{Acknowledgments}\vspace{-1.5mm}
\noindent{I thank the referee of this paper for his/her helpful suggestions concerning
the exposition in this work. This work is supported by a scholarship from the Indian Institute of Science and by a
UGC CAS-II grant (Grant No. F.510/25/CAS-II/2018(SAP-I)).}

\end{document}